\documentclass[12pt]{article}
\usepackage{e-jc}

\usepackage{amsmath,amssymb}
\usepackage{graphicx}
\usepackage[colorlinks=true,citecolor=black,linkcolor=black,urlcolor=blue]{hyperref}

\usepackage{tikz}
\tikzset{
	element/.style={circle,fill=black,scale=0.5}
}
\usetikzlibrary{calc}

\usepackage{enumerate}

\usepackage{color}
\newcommand{\blue}{\textcolor{black}} 

\newcommand{\orange}{\textcolor{black}}

\dateline{October 5, 2020}{May 4, 2021}{TBD}

\Copyright{The authors. Released under the CC BY-ND license (International 4.0).}
\usepackage{latexsym}
\usepackage{epsfig}
\usepackage{epic}
\newtheorem{sublemma}{}[theorem]

\newcommand{\cl}{{\rm cl}}
\newcommand{\si}{{\rm si}}
\newcommand{\co}{{\rm co}}

\newcommand{\delete}{\backslash}

\newcommand{\R}{\mathbb{R}}

\title{Elastic Elements in $3$-Connected Matroids}

\author{George Drummond\\
	\small School of Mathematics and Statistics\\[-0.8ex]
	\small University of Canterbury\\[-0.8ex] 
	\small  Christchurch, New Zealand\\
	\small\tt george.drummond@pg.canterbury.ac.nz\\
	\and
	Zach Gershkoff\\
	\small Mathematics Department\\[-0.8ex]
	\small Louisiana State University\\[-0.8ex]
	\small  Baton Rouge, Louisiana, USA\\
	\small\tt zgersh2@lsu.edu\\
	\and
	Susan Jowett\\
	\small School of Mathematics and Statistics\\[-0.8ex]
	\small Victoria University of Wellington\\[-0.8ex]
	\small Wellington, New Zealand\\
	\small\tt susan.jowett@vuw.ac.nz\\
	\and
	Charles Semple\\
	\small School of Mathematics and Statistics\\[-0.8ex]
	\small University of Canterbury\\[-0.8ex]
	\small Christchurch, New Zealand\\
	\small\tt charles.semple@canterbury.ac.nz\\
	\and 
	Jagdeep Singh\\
	\small Mathematics Department\\[-0.8ex]
	\small Louisiana State University\\[-0.8ex]
	\small  Baton Rouge, Louisiana, USA\\
	\small\tt jsing29@lsu.edu}

\begin{document}


\maketitle

\begin{abstract}
It follows by Bixby's Lemma that if $e$ is an element of a $3$-connected matroid $M$, then either $\co(M\delete e)$, the cosimplification of $M\delete e$, or $\si(M/e)$, the simplification of $M/e$, is $3$-connected. A natural question to ask is whether $M$ has an element $e$ such that both $\co(M\delete e)$ and $\si(M/e)$ are $3$-connected. Calling such an element ``elastic'', in this paper we show that if $|E(M)|\ge 4$, then $M$ has at least four elastic elements provided $M$ has no $4$-element fans \blue{and, up to duality, $M$ has no $3$-separating set $S$ that is the disjoint union of a rank-$2$ subset and a corank-$2$ subset of $E(M)$ such that $M|S$ is isomorphic to a member or a single-element deletion of a member of a certain family of matroids}.
\end{abstract}


\section{Introduction}

A result widely used in the study of $3$-connected matroids is due to Bixby~\cite{Bixby}: if $e$ is an element of a $3$-connected matroid $M$, then either $M\delete e$ or $M/e$ has no non-minimal $2$-separations, in which case, $\co(M\delete e)$, the cosimplification of $M$, or $\si(M/e)$, the simplification of $M$, is $3$-connected. A $2$-separation $(X, Y)$ is {\em minimal} if $\min \{|X|, |Y|\}=2$. This result is commonly referred to as Bixby's Lemma. Thus, although an element $e$ of a $3$-connected matroid $M$ may have the property that neither $M\delete e$ nor $M/e$ is $3$-connected, Bixby's Lemma says that at least one of $M\delete e$ and $M/e$ is close to being $3$-connected in a very natural way. In this paper, we are interested in whether or not there are elements $e$ in $M$ such that both $\co(M\delete e)$ and $\si(M/e)$ are $3$-connected, in which case, we say $e$ is {\em elastic}. In general, a $3$-connected matroid $M$ need not have any elastic elements. For example, all wheels and whirls of rank at least four have no elastic elements. The reason for this is that every element of such a matroid is in a $4$-element fan and, geometrically, every $4$-element fan is positioned in \blue{a certain way relative} to the rest of the elements of the matroid. \blue{However, $4$-element fans are not the only obstacles to $M$ having elastic elements.}

\blue{Let $n\ge 3$, and let $Z=\{z_1, z_2, \ldots, z_n\}$ be a basis of $PG(n-1, \R)$. Suppose that $L$ is a line that is freely placed relative to $Z$. For each $i\in \{1, 2, \ldots, n\}$, let $w_i$ be the unique point of $L$ contained in the hyperplane spanned by $Z-\{z_i\}$. Let $W=\{w_1, w_2, \ldots, w_n\}$, and let $\Theta_n$ denote the restriction of $PG(n-1, \R)$ to $W\cup Z$. Note that \blue{$\Theta_n$ is $3$-connected and} $Z$ is a corank-$2$ subset of $\Theta_n$. For all $i\in \{1, 2, \ldots, n\}$, we denote the matroid $\Theta_n\delete w_i$ by $\Theta^-_n$. The matroid $\Theta^-_n$ is well defined as, up to isomorphism, $\Theta_n\delete w_i\cong \Theta_n\delete w_j$ for all $i, j\in \{1, 2, \ldots, n\}$. \blue{For the interested reader,} the matroid $\Theta_n$ underlies the matroid operation of segment-cosegment exchange~\cite{oxl00} which generalises the operation of delta-wye exchange. A more formal definition of $\Theta_n$ is given in Section~\ref{thetasep}.}

\blue{If $n=3$, then $\Theta_3$ is isomorphic to $M(K_4)$. However, for all $n\geq 4$, the matroid $\Theta_n$ has no $4$-element fans and, also, no elastic elements.} \blue{Furthermore, for all $n\ge 3$, the set $W$ is a modular flat of $\Theta_n$~\cite{oxl00}. Thus, if $M$ is a matroid and $W$ is a subset of $E(M)$ such that $M|W\cong U_{2, n}$, then the generalised parallel connection $P_W(\Theta_n, M)$ of $\Theta_n$ and $M$ exists. In particular, it is straightforward to construct $3$-connected matroids having no $4$-element fans and no elastic elements. For example, take $U_{2, n}$ and repeatedly use the generalised parallel connection to ``attach'' copies of $\Theta_k$, where $4\le k\le n$, to any $k$-element subset of the elements of $U_{2, n}$.}



\blue{Let $M$ be a $3$-connected matroid, and let $A$ and $B$ be rank-$2$ and corank-$2$ subsets of $E(M)$. We say that $A\cup B$ is a {\em $\Theta$-separator} of $M$ if \blue{$r(M)\ge 4$ and $r^*(M)\ge 4$, and} either $M|(A\cup B)$ or $M^*|(A\cup B)$ is isomorphic to one of the matroids $\Theta_n$ and $\Theta^-_n$ for some $n\ge 3$.} \blue{We will show in Section~\ref{thetasep} that if $S$ is a $\Theta$-separator of $M$, then $S$ contains at most one elastic element.} \blue{Note that if $r(M)=3$, then $\si(M/e)$ is $3$-connected for all $e\in E(M)$, while if $r^*(M)=3$, then $\co(M\delete e)$ is $3$-connected for all $e\in E(M)$.} \blue{The main theorem of this paper is that, alongside $4$-element fans, $\Theta$-separators are the only obstacles to elastic elements in $3$-connected matroids.}

A $3$-separation $(A, B)$ of a matroid is {\em vertical} if $\min\{r(A), r(B)\}\ge 3$. Now, let $M$ be a matroid and let $(X, \{e\}, Y)$ be a partition of $E(M)$. We say that $(X, \{e\}, Y)$ is a {\em vertical $3$-separation} of $M$ if $(X\cup \{e\}, Y)$ and $(X, Y\cup \{e\})$ are both vertical $3$-separations and $e\in \cl(X)\cap \cl(Y)$. Furthermore, $Y\cup \{e\}$ is \emph{maximal} in this separation if there exists no vertical $3$-separation $(X', \{e'\}, Y')$ of $M$ such that $Y\cup \{e\}$ is a proper subset of $Y'\cup \{e'\}$. Essentially, all of the work in the paper goes into establishing the following theorem.

\begin{theorem}
\label{main1}
Let $M$ be a $3$-connected matroid with a vertical $3$-separation $(X, \{e\}, Y)$ such that $Y\cup \{e\}$ is maximal. Then at least one of the following holds:
\begin{enumerate}[{\rm (i)}]
\item $X$ contains at least two elastic elements;

\item $X\cup \{e\}$ is a $4$-element fan; \blue{or}

\item \blue{$X$ is contained in a $\Theta$-separator}.
\end{enumerate}
\end{theorem}

Note that, in the context of Theorem~\ref{main1}, if $X\cup \{e\}$ is a $4$-element fan, then it is possible that $X$ contains two elastic elements. For example, consider the rank-$4$ matroids $M_1$ and $M_2$ for which geometric representations are shown in Fig.~\ref{fig: elastic fans}. For each $\blue{i}\in \{1, 2\}$, the tuple $F=(e_1, e_2, e_3, e_4)$ is a $4$-element fan of $M_i$ and $(F-\{e_1\}, \{e_1\}, E(M_i)-F)$ is a vertical $3$-separation of $M_i$. In $M_1$, none of $e_2$, $e_3$, and $e_4$ \blue{are} elastic, while in $M_2$, both $e_2$ and $e_3$ are elastic. However, provided $X\cup \{e\}$ is a maximal fan, the instance illustrated in Fig.~\ref{fig: elastic fans}(i) is essentially the only way in which $X$ does not contain two elastic elements. This is made more precise in Section~\ref{fansplus}. \blue{\blue{As noted above,} if $X$ is contained in a $\Theta$-separator, then $X$ contains at most one elastic element. The details of the way in which this happens is given in Section~\ref{thetasep}.}

\begin{figure}[ht]
	\centering
	\hspace{2cm}
	\begin{minipage}{0.4\textwidth}
		\begin{tikzpicture}[scale=1]
		\coordinate (cntrl1) at (-50:2.5) {};
		\coordinate (cntrl2) at (-20:2.5) {};
		\begin{scope}[every node/.style=element]
		\coordinate (c1) at (-210:2) {};
		\coordinate (c2) at (90:1) {};
		\node (e1) at (30:2) {};
		\node (e3) at (0,4) {};
		\coordinate (b) at (-120:2) {};

		\draw (c1) to (e1);
		\node (e2) at ($(e1)!0.5!(e3)$) {};

		\draw (c1) to node[pos=0.3] {} node[pos=0.6] {} (b);
		\draw (e1) to node[pos=0.3] {} node[pos=0.6] {} (b);
		\draw (e1) to (e2) to (e3);
		\draw (e2) to (c1);
		\draw (e3) to (c2);
		
		\node (e4) at (intersection of e3--c2 and e2--c1) {};
		
		\end{scope}
		\node[scale=1.0] at ($(e1)+(30:0.5)$) {$e_1$};
		\node[scale=1.0] at ($(e2)+(30:0.5)$) {$e_3$};
		\node[scale=1.0] at ($(e3)+(30:0.5)$) {$e_2$};
		\node[scale=1.0] at ($(e4)+(-30:0.5)$) {$e_4$};
		\node[scale=1.0] at (-90:2.5) {(i) $M_1$};
		\end{tikzpicture}
	\end{minipage}
	\begin{minipage}{0.4\textwidth}
		\begin{tikzpicture}[scale=1]
		\coordinate (cntrl1) at (-50:2.5) {};
		\coordinate (cntrl2) at (-20:2.5) {};
		\begin{scope}[every node/.style=element]
		\coordinate (c1) at (-210:2) {};
		\coordinate (c2) at (90:1) {};
		\coordinate (i1) at ($(c1)!0.5!(c2)$) {};
		\node (e1) at (30:2) {};
		\node (e3) at (0,4) {};
		\coordinate (b) at (-120:2) {};

		\draw (c1) to (e1);
		\node (e2) at ($(e1)!0.5!(e3)$) {};

		\draw (c1) to node[pos=0.3] {} node[pos=0.6] {} (b);
		\draw (e1) to node[pos=0.3] {} node[pos=0.6] {} (b);
		\draw (e1) to (e2) to (e3);
		\draw (e2) to (i1);
		\draw (e3) to (c2);
		
		\node (e4) at (intersection of e3--c2 and e2--i1) {};
		
		\end{scope}
		\node[scale=1.0] at ($(e1)+(30:0.5)$) {$e_1$};
		\node[scale=1.0] at ($(e2)+(30:0.5)$) {$e_3$};
		\node[scale=1.0] at ($(e3)+(30:0.5)$) {$e_2$};
		\node[scale=1.0] at ($(e4)+(-30:0.5)$) {$e_4$};
		\node[scale=1.0] at (-90:2.5) {(ii) $M_2$};
		\end{tikzpicture}
	\end{minipage}
	\caption{For each $i\in \{1, 2\}$, the tuple $(e_1, e_2, e_3, e_4)$ is a $4$-element fan and the partition $(\{e_2, e_3, e_4\}, \{e_1\}, E(M_i)-\{e_1, e_2, e_3, e_4\})$ of $E(M_i)$ is a vertical $3$-separation of $M_i$. Furthermore, in $M_1$, none of $e_2$, $e_3$, and $e_4$ are elastic, while in $M_2$, both $e_2$ and $e_3$ are elastic.}
	\label{fig: elastic fans}
\end{figure}

%
%
%
%
%

An almost immediate consequence of Theorem~\ref{main1} is the following corollary.

\begin{corollary}
\label{main2}
Let $M$ be a $3$-connected matroid. If $|E(M)|\ge 7$, then $M$ contains at least four elastic elements provided $M$ has no $4$-element fans \blue{and no $\Theta$-separators}. Moreover, if $|E(M)|\leq 6$, then every element of $M$ is elastic.
\end{corollary}

\blue{The condition in Corollary~\ref{main2} that $M$ has no $4$-element fans and no $\Theta$-separators is not necessarily that restrictive. For example, if $N$ is an excluded minor for $GF(q)$-representability (or, more generally, for $\mathbb P$-representability, where $\mathbb P$ is a partial field), then $N$ has no $4$-element fans and no $\Theta$-separators. The fact that $N$ has no $4$-element fans is well known and straightforward to show. To see that $N$ has no $\Theta$-separators, suppose that $N$ has a $\Theta$-separator. By duality, we may assume that $N$ has rank-$2$ and corank-$2$ sets $W$ and $Z$, respectively, such that $M|(W\cup Z)$ is isomorphic to either $\Theta_n$ or $\Theta^-_n$, for some $n\ge 3$. Say $M|(W\cup Z)$ is isomorphic to $\Theta_n$. Then the matroid $N'$ obtained from $N$ by a cosegment-segment exchange on $Z$ is isomorphic to the matroid obtained from $N$ by deleting $Z$ and, for each $w\in W$, adding an element in parallel to $w$. It is shown in \cite[Theorem~1.1]{oxl00} that the class of excluded minors for $GF(q)$-representability (or, more generally, $\mathbb P$-representability) is closed under the operation of cosegment-segment exchange, and so $N'$ is also an excluded minor for $GF(q)$-representability. But $N'$ contains elements in parallel, a contradiction. The same argument holds if $M|(W\cup Z)$ is isomorphic to $\Theta^-_n$ except that, in applying a cosegment-segment exchange, we additionally add an element freely in the span of $W$.}


Like Bixby's Lemma, Corollary~\ref{main2} is an inductive tool for handling the removal of elements of $3$-connected matroids while preserving connectivity. The most well-known examples of such tools are Tutte's Wheels-and-Whirls Theorem~\cite{tut66} and Seymour's Splitter Theorem~\cite{sey80}. In both theorems, this removal preserves $3$-connectivity. More recently, there have been analogues of these theorems in which the removal of elements preserves $3$-connectivity up to simplification and cosimplification. These analogues have additional conditions on the elements being removed. Let $B$ be a basis of a $3$-connected matroid $M$, and suppose that $M$ has no $4$-element fans. Say $M$ is representable over some field $\mathbb F$ and that we are given a standard representation of $M$ over $\mathbb F$. To keep the information displayed by the representation in an $\mathbb F$-representation of a single-element deletion or a single element contraction of $M$, we need to avoid pivoting. To do this, we want to either contract an element in $B$ or delete an element in $E(M)-B$. Whittle and Williams~\cite{Whittle&Williams} showed that if $|E(M)|\ge 4$, then $M$ has at least four elements $e$ such that either $\si(M/e)$ is $3$-connected if $e\in B$ or $\co(M\delete e)$ is $3$-connected if $e\in E(M)-B$. Brettell and Semple~\cite{Brettell&Semple} establish a Splitter Theorem counterpart to this last result where, again, $3$-connectivity is preserved up to simplification and cosimplification. These last two results are related to an earlier result of Oxley et al.~\cite{Oxley et al}. Indeed, the starting point for the proof of Theorem~\ref{main1} is~\cite{Oxley et al}. 

%
%
%
%
%
%
%

The paper is organised as follows. The next section contains some necessary preliminaries on connectivity, while \orange{Section~\ref{fansplus} considers fans and determines exactly which elements of a fan are elastic}. Section~\ref{sec: segments} establishes two results concerning when an element in a rank-$2$ restriction of a $3$-connected matroid is deletable or contractible, \blue{and Section~\ref{thetasep} considers $\Theta$-separators, and determines the elasticity of the elements of these sets.} Section~\ref{proof} consists of the proofs of Theorem~\ref{main1} and Corollary~\ref{main2}. \blue{Effectively, all of the work that proves these two results goes into proving Theorem~\ref{main1}. We break the proof of Theorem~\ref{main1} into two lemmas depending on whether or not $X$ contains at least one element that is not contractible. The statements of these lemmas, Lemma~\ref{key1} and Lemma~\ref{key2}, provide additional structural information when $X$ is contained in a $\Theta$-separator.} Throughout the paper, the notation and terminology follows~\cite{ox11}.

\section{Preliminaries}
\label{sec: prelims}

\subsection*{Connectivity}

Let $M$ be a matroid with ground set $E$ and rank function~$r$. The \emph{connectivity function} $\lambda_M$ of $M$ is defined on all subsets $X$ of $E$ by
$$\lambda_M(X)=r(X)+r(E-X)-r(M).$$
\blue{Equivalently, $\lambda_M(X)=r(X)+r^*(X)-|X|$.} A subset $X$ of $E$ or a partition $(X, E-X)$ is \emph{$k$-separating} if $\lambda_M(X)\le k-1$ and {\em exactly $k$-separating} if $\lambda_M(X)=k-1$. A $k$-separating partition $(X, E-X)$ is a \emph{$k$-separation} if $\min \{|X|, |E-X|\}\geq k$. A matroid is \emph{$n$-connected} if it has no $k$-separations for all $k < n$.

%

Let $e$ be an element of a $3$-connected matroid $M$. We say $e$ is {\em deletable} if $\co(M\delete e)$ is $3$-connected, and $e$ is {\em contractible} if $\si(M/e)$ is $3$-connected. Thus, $e$ is elastic if it is both deletable and contractible. 

Two $k$-separations $(X_1, Y_1)$ and $(X_2, Y_2)$ \emph{cross} if each of the intersections $X_1\cap Y_1$, $X_1\cap Y_2$, $X_2\cap Y_1$, $X_2\cap Y_2$ are non-empty. The next lemma is a standard tool for dealing with crossing separations. It is a straightforward consequence of the fact that the connectivity function $\lambda$ of a matroid $M$ is submodular, that is,
$$\lambda(X) + \lambda(Y)\ge \lambda(X\cap Y) + \lambda(X\cup Y)$$
for all $X, Y\subseteq E(M)$. An application of this lemma will be referred to as \emph{by uncrossing}.

\begin{lemma}
\label{uncrossing}
Let $M$ be a $k$-connected matroid, and let $X$ and $Y$ be $k$-separating subsets of $E(M)$.
\begin{enumerate}[{\rm (i)}]
\item If $|X\cap Y|\geq k-1$, then $X\cup Y$ is $k$-separating.

\item If $|E(M)-(X\cup Y)|\geq k-1$, then $X\cap Y$ is $k$-separating.
\end{enumerate}
\end{lemma}

The next five lemmas are used frequently throughout the paper. The first follows from orthogonality, while the second follows from the first. The third follows from the first and second. A proof of the fourth and fifth can be found in~\cite{whi99} and~\cite{Brettell&Semple}, respectively.

\begin{lemma}
\label{orthogonality}
Let $e$ be an element of a matroid $M$, and let $X$ and $Y$ be disjoint sets whose union is $E(M)-\{e\}$. Then $e\in \cl(X)$ if and only if $e\not\in \cl^*(Y)$.
\end{lemma}

\begin{lemma}
\label{3sep1}
Let $X$ be an exactly $3$-separating set in a $3$-connected matroid $M$, and suppose that $e\in E(M)-X$. Then $X\cup \{e\}$ is $3$-separating if and only if $e\in \cl(X)\cup \cl^*(X)$.
\end{lemma}

\begin{lemma}
Let $(X, Y)$ be an exactly $3$-separating partition of a $3$-connected matroid $M$, and suppose that $|X|\ge 3$ and $e\in X$. Then $(X-\{e\}, Y\cup \{e\})$ is exactly $3$-separating if and only if $e$ is in exactly one of $\cl(X-\{e\})\cap \cl(Y)$ and $\cl^*(X-\{e\})\cap \cl^*(Y)$.
\label{3sep2}
\end{lemma}

\begin{lemma}
\label{triangle}
Let $C^*$ be a rank-$3$ cocircuit of a $3$-connected matroid $M$. If $e\in C^*$ has the property that $\cl(C^*)-\{e\}$ contains a triangle of $M/e$, then $\si(M/e)$ is $3$-connected.
\end{lemma}

\begin{lemma}
\label{notmany}
Let $(X, Y)$ be a $3$-separation of a $3$-connected matroid $M$. If $X\cap \cl(Y)\neq \emptyset$ and $X\cap \cl^*(Y)\neq \emptyset$, then $|X\cap \cl(Y)|=|X\cap \cl^*(Y)|=1$.
\end{lemma}

\subsection*{Vertical connectivity}

A $k$-separation $(X, Y)$ of a matroid $M$ is \emph{vertical} if $\min\{r(X), r(Y)\}\geq k$. As noted in the introduction, we say a partition $(X, \{e\}, Y)$ of $E(M)$ is a \emph{vertical $3$-separation} of $M$ if $(X\cup \{e\}, Y)$ and $(X, Y\cup \{e\})$ are both vertical $3$-separations of $M$ and $e\in \cl(X)\cap \cl(Y)$. Furthermore, $Y\cup \{e\}$ is {\em maximal} if there is no vertical $3$-separation $(X', \{e'\}, Y')$ of $M$ such that $Y\cup \{e\}$ is a proper subset of $Y'\cup \{e'\}$. A $k$-separation $(X, Y)$ of $M$ is {\em cyclic} if both $X$ and $Y$ contain circuits. The next lemma gives a duality link between the cyclic $k$-separations and vertical $k$-separations of a $k$-connected matroid.

\begin{lemma}
Let $(X, Y)$ be a partition of the ground set of a $k$-connected matroid $M$. Then $(X, Y)$ is a cyclic $k$-separation of $M$ if and only if $(X, Y)$ is a vertical $k$-separation of $M^*$.
\label{dual}
\end{lemma}

\begin{proof}
Suppose that $(X, Y)$ is a cyclic $k$-separation of $M$. Then $(X, Y)$ is a $k$-separation of $M^*$. Since $(X, Y)$ is a $k$-separation of a $k$-connected matroid, $(X, Y)$ is exactly $k$-separating, and so $r(X)+r(Y)-r(M)=k-1$. Therefore, as $r^*(X)=r(Y)+|X|-r(M)$, it follows that
$$r^*(X)=((k-1)-r(X)+r(M))+|X|-r(M)=(k-1)+|X|-r(X).$$
As $X$ contains a circuit, $X$ is dependent, so $|X|-r(M)\ge 1$. Hence $r^*(X)\ge k$. By symmetry, $r^*(Y)\ge k$, and so $(X, Y)$ is a vertical $k$-separation of $M^*$. A similar argument establishes the converse.
\end{proof}

Following Lemma~\ref{dual}, we say a partition $(X, \{e\}, Y)$ of the ground set of a $3$-connected matroid $M$ is a \emph{cyclic $3$-separation} if $(X, \{e\}, Y)$ is a vertical $3$-separation of $M^*$.


Of the next two results, the first combines Lemma~\ref{dual} with a straightforward strengthening of~ \cite[Lemma 3.1]{Oxley et al} and, in combination with Lemma~\ref{dual}, the second follows easily from Lemma~\ref{3sep2}.

\begin{lemma}
\label{vertical1}
Let $M$ be a $3$-connected matroid, and suppose that $e\in E(M)$. Then $\si(M/e)$ is not $3$-connected if and only if $M$ has a vertical $3$-separation $(X, \{e\}, Y)$. Dually, $\co(M\delete e)$ is not $3$-connected if and only if $M$ has a cyclic $3$-separation $(X, \{e\}, Y)$.
\end{lemma}

\begin{lemma}
\label{vertical2}
Let $M$ be a $3$-connected matroid. If $(X, \{e\}, Y)$ is a vertical $3$-separation of $M$, then $(X -\cl(Y), \{e\}, \cl(Y)-e)$ is also a vertical $3$-separation of $M$. Dually, if $(X, \{e\}, Y)$ is a cyclic $3$-separation of $M$, then $(X-\cl^*(Y), \{e\}, \cl^*(Y)-\{e\})$ is also a cyclic $3$-separation of $M$.
\end{lemma}

\noindent Note that an immediate consequence of Lemma~\ref{vertical2} is that if $(X, \{e\}, Y)$ is a vertical $3$-separation such that $Y\cup \{e\}$ is maximal, then $Y\cup \{e\}$ must be closed. We will make repeated use of this fact.

\section{Fans}
\label{fansplus}

Let $M$ be a $3$-connected matroid. A subset $F$ of $E(M)$ with at least three elements is a \emph{fan} if there is an ordering $(f_1, f_2, \ldots, f_k)$ of $F$ such that
\begin{enumerate}[(i)]
\item for all $i\in \{1, 2, \ldots, k-2\}$, the triple $\{f_i, f_{i+1}, f_{i+2}\}$ is either a triangle or a triad, and

\item for all $i\in \{1, 2, \ldots, k-3\}$, if $\{f_i, f_{i+1}, f_{i+2}\}$ is a triangle, then $\{f_{i+1}, f_{i+2}, f_{i+3}\}$ is a triad, while if $\{f_i, f_{i+1}, f_{i+2}\}$ is a triad, then $\{f_{i+1}, f_{i+2}, f_{i+3}\}$ is a triangle.
\end{enumerate}
If $k\geq 4$, then the elements $f_1$ and $f_k$ are the \emph{ends} of $F$. Furthermore, if $\{f_1, f_2, f_3\}$ is a triangle, then $f_1$ is a {\em spoke-end}; otherwise, $f_1$ is a {\em rim-end}. Observe that if $F$ is a $4$-element fan $(f_1, f_2, f_3, f_4)$, then either $f_1$ or $f_4$ is the unique spoke-end of $F$ depending on whether $\{f_1, f_2, f_3\}$ or $\{f_2, f_3, f_4\}$ is a triangle, respectively. The proof of the next lemma is straightforward and omitted.



%

\begin{lemma}
\label{fans & seps}
Let $M$ be a $3$-connected matroid, and suppose that $F=(f_1, f_2, f_3, f_4)$ is a $4$-element fan of $M$ with spoke-end $f_1$. Then $(\{f_2, f_3, f_4\}, \{f_1\}, E(M)-F)$ is a vertical $3$-separation of $M$ provided $r(M)\ge 4$, in which case, $E(M)-\{f_2, f_3, f_4\}$ is maximal.
\end{lemma}

\orange{We end this section by determining when an element in a fan of size at least four is elastic. For subsets $X$ and $Y$ of a matroid, the {\em local connectivity} between $X$ and $Y$, denoted $\sqcap(X, Y)$, is defined by
$$\sqcap(X, Y)=r(X)+r(Y)-r(X\cup Y).$$
Let $M$ be a $3$-connected matroid and let $k$ be a positive integer. A {\em flower} $\Phi$ of $M$ is an (ordered) partition $(P_1, P_2, \ldots, P_k)$ of $E(M)$ such that each $P_i$ has at least two elements and is $3$-separating, and each $P_i\cup P_{i+1}$ is $3$-separating, where all subscripts are interpreted modulo $k$. If $k\ge 4$, we say $\Phi$ is {\em swirl-like} if $\bigcup_{i\in I} P_i$ is exactly $3$-separating for all proper subsets $I$ of $\{1, 2, \ldots, k\}$ whose members form a consecutive set in the cyclic order $(1, 2, \ldots, k)$, and
$$\sqcap(P_i, P_j)=
\begin{cases}
1, & \mbox{if $P_i$ and $P_j$ are consecutive}; \\
0, & \mbox{if $P_i$ and $P_j$ are not consecutive}
\end{cases}
$$
for all distinct $i, j\in \{1, 2, \ldots, k\}$. For further details of swirl-like flowers and, more generally flowers, we refer the reader to~\cite{oxl04}.}

\begin{lemma}
\label{elastic fans}
Let $M$ be a $3$-connected matroid such that $r(M),r^*(M)\geq 4$, and let $F=(f_1, f_2, \ldots, f_n)$ be a maximal fan of $M$.
\begin{enumerate}[{\rm (i)}]
\item If $n\ge 6$, then $F$ contains no elastic elements of $M$.
\item If $n=5$, then $F$ contains either exactly one elastic element, namely $f_3$, or no elastic elements of $M$.
\item If $n=4$, then $F$ contains either exactly two elastic elements, namely $f_2$ and $f_3$, or no elastic elements of $M$.
\end{enumerate}
Moreover, if $n\in \{4, 5\}$ and $F$ contains no elastic elements, then, \blue{up to duality,} $M$ has a swirl-like flower $(A, \{f_1, f_2\}, F-\{f_1, f_2\}, B)$ as shown geometrically in Fig.~\ref{fig: flower}, \blue{or $n=5$ and there is an element $g$ such that $M|(F\cup\{g\})\cong M(K_4)$.}
\end{lemma}

\begin{proof}
It follows by Lemma~\ref{fans & seps} that the ends of a $4$-element fan in $M$ are not elastic. Thus, if $n\ge 6$, then, as every element of $F$ is the end of a $4$-element fan, $F$ contains no elastic elements, and if $n=5$, then, as every element of $F$, except $f_3$, is the end of a $4$-element fan, $F$ contains no elastic elements except possibly $f_3$. Thus (i) and (ii) hold, \blue{and we assume that $n\in \{4, 5\}$}. \blue{By applying the dual argument if needed, we may also assume that $\{f_1, f_2, f_3\}$ is a triangle.}


\begin{sublemma}
\blue{If $f_3$ is contractible, then $f_3$ is elastic unless $n=5$ and there is an element $g$ such that $M|(F\cup \{g\})\cong M(K_4)$, or $n=4$ and $f_2$ is not contractible.}
\label{littlefan}
\end{sublemma}

\blue{Suppose that $f_3$ is contractible. If $f_3$ is not elastic, then $\co(M\delete f_3)$ is not $3$-connected. First assume that $n=5$. Then, as $f_2$ is the end of a $4$-element fan, $\co(M\delete f_2)$ is not $3$-connected, and so, by Bixby's Lemma, $\si(M/f_2)$ is $3$-connected. By orthogonality, $\{f_2, f_3, f_4\}$ is the unique triad containing $f_3$, and so $\co(M\delete f_3)\cong M/f_2\delete f_3$. But then $\co(M\delete f_3)$ is $3$-connected unless there is an element $g$ such that $\{f_2, f_4, g\}$ is a triangle of $M$, in which case $M|(F\cup \{g\})\cong M(K_4)$. Now assume that $n=4$. If $f_3$ is contained in a triad $T^*$ other than $\{f_2, f_3, f_4\}$, then, by orthogonality, either $f_1$ or $f_2$ is contained in $T^*$. If $f_1\in T^*$, then $F$ is not maximal, a contradiction. Thus $f_2\in T^*$. But then $T^*\cup \{f_4\}$ has corank~$2$ and so, as $M$ is $3$-connected, $(T^*\cup \{f_4\})-\{f_2\}$ is a triad, contradicting orthogonality. Thus, as $F$ is maximal, $\{f_2, f_3, f_4\}$ is the unique triad containing $f_3$. Hence $\co(M\delete f_3)\cong M/f_2\delete f_3$. Thus $\co(M\delete f_3)\cong \si(M/f_2)$ and so, as $\co(M\delete f_3)$ is not $3$-connected, $f_2$ is not contractible. This completes the proof of (\ref{littlefan}).}


Since $(f_1, f_3, f_2, f_4)$ is also a fan ordering for $F$ \blue{if $n=4$, it follows by (\ref{littlefan}) that we may now assume $\si(M/f_3)$ is not $3$-connected.} We next \blue{complete the proof of} the lemma for when $n=4$. The remaining part of the lemma for when $n=5$ is proved similarly \blue{and is omitted}.

\blue{As $\si(M/f_3)$ is not $3$-connected, it follows} by Lemma~\ref{vertical1} \blue{that}
$$(A\cup \{f_1, f_2\}, \{f_3\}, B\cup \{f_4\})$$
is a vertical $3$-separation of $M$, where $|A|\ge 1$ and $|B|\ge 2$. Say $|A|=1$, where $A=\{f_0\}$. Then $A\cup \{f_1, f_2\}$ is a triad, and so $(f_0, f_1, f_2, f_3, f_4)$ is a $5$-element fan, contradicting the maximality of $F$. Thus $|A|\ge 2$. Since $A\cup B$ and $B\cup \{f_4\}$ are $3$-separating in $M$, it follows by uncrossing that $B$ is $3$-separating in $M$. Similarly, $A$ is $3$-separating in $M$. Hence
$$(A, \{f_1, f_2\}, \{f_3, f_4\}, B)$$ is a flower $\Phi$. Since $\sqcap(\{f_1, f_2\}, \{f_3, f_4\})=1$, it follows by \cite[Theorem~4.1]{oxl04} that
$$\sqcap(A, \{f_1, f_2\})=\sqcap(\{f_3, f_4\}, B)=\sqcap(A, B)=1.$$
To show that $\Phi$ is a swirl-like flower, it remains to show that
$$\sqcap(\{A, \{f_3, f_4\})=\sqcap(B, \{f_1, f_2\})=0.$$

If $f_1\not\in {\rm cl}(A)$, then, as $f_2\not\in {\rm cl}(A\cup \{f_1\})$, it follows that $r(A\cup \{f_1, f_2\})=r(A)+2$. But then $\sqcap(A, \{f_1, f_2\})=0$, a contradiction. Thus $f_1\in {\rm cl}(A)$. Furthermore, $f_3\not\in {\rm cl}(A)$. Assume that $f_4\in {\rm cl}(A\cup \{f_3\})$. Then, as $\sqcap(\{f_3, f_4\}, B)=1$,
\begin{align*}
1 & = r_{M/f_3}(A\cup \{f_1, f_2\})+r_{M/f_3}(B\cup \{f_4\})-r(M/f_3) \\
& = r_{M/f_3}(A\cup \{f_1, f_2, f_4\})+r_{M/f_3}(B)-r(M/f_3) \\
& = r(A\cup F)-1+r(B)-(r(M)-1) \\
& = r(A\cup F)+r(B)-r(M),
\end{align*}
and so $B$ is $2$-separating in $M$, a contradiction. Thus $f_4\not\in {\rm cl}(A\cup \{f_3\})$, and so $\sqcap(A, \{f_3, f_4\})=0$. To see that $\sqcap(B, \{f_1, f_2\})=0$, first assume that $f_1\in {\rm cl}(B)$. Then, as $f_1\in {\rm cl}(A)$,
\begin{align*}
1 & = r_{M/f_3}(A\cup \{f_1, f_2\})+r_{M/f_3}(B\cup \{f_4\})-r(M/f_3) \\
& = r_{M/f_3}(A)+r_{M/f_3}(B\cup \{f_1, f_2, f_4\})-r(M/f_3) \\
& = r(A)+r(B\cup F)-1-(r(M)-1) \\
& = r(A)+r(B\cup F)-r(M),
\end{align*}
and so $A$ is $2$-separating in $M$. This contradiction implies that $f_1\not\in {\rm cl}(B)$. It follows that $r(B\cup \{f_1, f_2\})=r(B)+2$, that is $\sqcap(B, \{f_1, f_2\})=0$. We deduce that $(A, \{f_1, f_2\}, \{f_3, f_4\}, B)$ is a swirl-like flower. \blue{Lastly, as $f_1\in \cl(A)$ and $\sqcap(B, \{f_3, f_4\})=1$, it follows that $(A\cup\{f_1\}  ,\{f_2\},B\cup\{f_3,f_4\})$ is a cyclic $3$-separation of $M$, and so $\co(M\delete f_2)$ is not $3$-connected, that is, $f_2$ is not elastic. Hence (iii) holds.}
\end{proof}

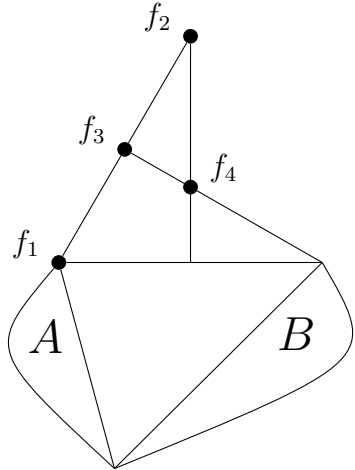
\begin{figure}[ht]
	\centering
	\begin{tikzpicture}[scale=1]
	\coordinate (cntrl1) at (-50:2.5) {};
	\coordinate (cntrl2) at (-20:2.5) {};
	\begin{scope}[every node/.style=element]
	\node (e1) at (-210:2) {};
	\coordinate (c2) at (90:1) {};
	\coordinate (c1) at (30:2);
	\node (e3) at (0,4) {};
	\coordinate (b) at (-120:2) {};
	
	\coordinate (cntr1) at ($(e1) + (230:1.5)$);
	\coordinate (cntr2) at ($(c1) + (-60:1.5)$);
	

	\draw (c1) to (e1);
	\node (e2) at ($(e1)!0.5!(e3)$) {};

	\draw (c1) to  (b);
	\draw (e1) to (b);
	
	\draw (e1) .. controls (cntr1) .. (b);
	\draw (c1) .. controls (cntr2) .. (b);
	\draw (e1) to (e2) to (e3);
	\draw (e2) to (c1);
	\draw (e3) to (c2);
	
	\node (e4) at (intersection of e3--c2 and e2--c1) {};
	
	\end{scope}
	\node[scale=1.0] at ($(e1)+(150:0.5)$) {$f_1$};
	\node[scale=1.0] at ($(e2)+(150:0.5)$) {$f_3$};
	\node[scale=1.0] at ($(e3)+(150:0.5)$) {$f_2$};
	\node[scale=1.0] at ($(e4)+(30:0.5)$) {$f_4$};
	
	\node[scale=1.5] at ($(e1)+(-100:1)$) {$A$};
	\node[scale=1.5] at ($(c1)+(-110:1)$) {$B$};
	
	\end{tikzpicture}
	\caption{The swirl-like flower $(A, \{f_1, f_2\}, F-\{f_1, f_2\}, B)$ of Lemma \ref{elastic fans} where, if $|F|=5$, then $f_5$ is an element in $B$.}
	\label{fig: flower}
\end{figure}

\section{Elastic Elements in Segments}
\label{sec: segments}

Let $M$ be a matroid. A subset $L$ of $E(M)$ of size at least two is a \emph{segment} if $M|L$ is isomorphic to a rank-$2$ uniform matroid. In this section we consider when an element in a segment is deletable or contractible. We begin with the following elementary lemma.

\begin{lemma}
\label{delete1}
Let $L$ be a segment of a $3$-connected matroid $M$. If $L$ has at least four elements, then $M\delete \ell$ is $3$-connected for all $\ell\in L$.
\end{lemma}

\noindent In particular, Lemma~\ref{delete1} implies that, in a $3$-connected matroid, every element of a segment with at least four elements is deletable. \blue{We next determine the structure which arises when elements of a segment in a $3$-connected matroid are not contractible.}

\begin{lemma}
\label{contract+}
Let $M$ be a $3$-connected matroid, and suppose that $L\cup \{w\}$ is a rank-$3$ cocircuit of $M$, where $L$ is a segment. \blue{If two distinct elements $y_1$ and $y_2$ of $L$ are not contractible, then there are distinct elements $w_1$ and $w_2$ of $E(M)-(L\cup\{w\})$ such that $(\cl(L)-\{y_i\})\cup\{w_i\}$ is a cocircuit for each $i\in \{1,2\}$.}
\end{lemma}

\begin{proof}
Let $y_1$ and $y_2$ be distinct elements of $L$ that are not contractible. For each $i\in \{1, 2\}$, it follows by Lemma~\ref{vertical1} that there exists a vertical $3$-separation $(X_i, \{y_i\}, Y_i)$ of $M$ such that $y_j\in Y_i$, where $\{i, j\}=\{1, 2\}$. By Lemma~\ref{vertical2}, we may assume $Y_i\cup \{y_i\}$ is closed, in which case, $L-\blue{\{}y_i\blue{\}}\subseteq Y_i$. \blue{Furthermore, for each $i\in \{1, 2\}$, we may also assume, amongst all such vertical $3$-separations of $M$, that $|Y_i|$ is minimised.} If $w\in Y_i$, then, as $L\cup \{w\}$ is a cocircuit, $X_i$ is contained in the hyperplane $E(M)-(L\cup \{w\})$, and so $y_i\not\in \cl(X_i)$. This contradiction implies that $w\in X_i$. Thus, for each $i\in \{1, 2\}$, we deduce that $M$ has a vertical $3$-separation
$$(U_i\cup \{w\}, \{y_i\}, V_i\cup (L-\blue{\{}y_i\blue{\}})),$$
where $U_i\cup \{w\}=X_i$ and $V_i\cup (L-\{y_i\})=Y_i$. Next we show the following.

\begin{sublemma}
\label{contract1.1}
For each $i\in \{1, 2\}$, we have $w\in \cl_M(U_i\cup \{y_i\})-\cl_M(U_i)$.
\end{sublemma}
Since $L\cup \{w\}$ is a cocircuit, the elements $y_i, w\not\in \cl_M(U_i)$. But $y_i\in \cl_M(U_i\cup \{w\})$, and so $y_i\in \cl_M(U_i\cup \{w\})-\cl_M(U_i)$. Thus, by the MacLane-Steinitz exchange property, $w\in \cl_M(U_i\cup \{y_i\})-\cl_M(U_i)$.

\begin{sublemma}
\label{contract1.2}
For each $i\in \{1, 2\}$, we have $y_i\not\in \cl_M(U_j\cup \{w\})$, where $\{i, j\}=\{1, 2\}$.
\end{sublemma}
By Lemma~\ref{vertical2},
$$(\cl(U_j\cup \{w\})-\{y_j\}, \{y_j\}, (V_j\cup (L-\blue{\{}y_j\blue{\}}))-\cl(U_j\cup \{w\}))$$
is a vertical $3$-separation of $M$. If $y_i \in \cl(U_j\cup \{w\})$, then, as $y_j\in \cl(U_j\cup \{w\})$, the segment $L$ is contained in $\cl(U_j\cup \{w\})$. Therefore $L\cup \{w\}\subseteq \cl(U_j\cup \{w\})$, and so $(V_j\cup (L-\blue{\{}y_{\blue{j}}\blue{\}}))-\cl(U_j\cup \{w\})=V_j-\cl(U_j\cup \{w\})$. Since $V_j-\cl(U_j\cup \{w\})$ is contained in the hyperplane $E(M)-(L\cup \{w\})$, it follows that $y_j\not\in V_j-\cl(U_j\cup \{w\})$, a contradiction. Thus~(\ref{contract1.2}) holds.

Since $M$ is $3$-connected and $(U_i\cup \{w\}, \{y_i\}, V_i\cup (L-\blue{\{}y_i\blue{\}}))$ is a vertical $3$-separation, it follows by~(\ref{contract1.1}) that
$$r(U_i)+r(V_i\cup L)-r(M\delete w)=r(U_i\cup \{w\})-1+r(V_i\cup L)-r(M)=1.$$
Thus $(U_i, V_i\cup L)$ is a $2$-separation of $M\delete w$ for each $i\in \{1, 2\}$. \blue{We next show that}

\begin{sublemma}
\label{1elt}
\blue{$|U_1\cap V_2|=|U_2\cap V_1|=1$.}
\end{sublemma}

\blue{Let $\{i, j\}=\{1, 2\}$. If $U_i\subseteq U_j$, then
$$y_i\in \cl(U_i\cup \{w\})\subseteq \cl(U_j\cup \{w\}),$$
contradicting~(\ref{contract1.2}). \blue{Therefore, for $\{i, j\}=\{1, 2\}$, we have $|U_i\cap V_j|\ge 1$.} Consider the $2$-connected matroid $M\delete w$. Since $|U_j\cap V_i|\ge 1$, it follows by uncrossing that $U_i\cup (V_j\cup L)$ is $2$-separating in $M\delete w$. But, by~(\ref{contract1.1}), $w\in \cl_M(U_i\cup L)$ and so $U_i\cup V_j\cup (L\cup \{w\})$ is $2$-separating in $M$. Since $M$ is $3$-connected, it follows that $|U_j\cap V_i|\leq 1$. Thus (\ref{1elt}) holds.
}

\blue{Let $w_1$ and $w_2$ be the unique elements of $U_2\cap V_1$ and $U_1\cap V_2$, respectively. Now $|(U_1\cup \{w\})\cap (U_2\cup \{w\})|\ge 2$ and so, by uncrossing, $V_1\cup L$ and $V_2\cup L$, as well as $V_1\cup L$ and $V_2\cup (L-\{\blue{y_1}\})$, we see that $(V_1\cap V_2)\cup L$ and $(V_1\cap V_2)\cup (L-\{y_1\})$ are $3$-separating in $M$. So
$$(\blue{U_1\cup U_2\cup \{w\}}, \{y_1\}, (V_1\cap V_2)\cup (L-\{y_1\}))$$
is a vertical $3$-separation of $M$ unless $r((V_1\cap V_2)\cup (L-\{y_1\})=2$. Since $V_1\cup L$ and $V_2\cup L$ are closed, $(V_1\cap V_2)\cup L$ is closed. Furthermore,
$$|(V_1\cap V_2)\cup (L-\{y_1\})| < |V_1\cup (L-\{y_1\})|,$$
and so, by the minimality of $|Y_1|$, we have $r((V_1\cap V_2)\cup (L-\{y_1\})=2$. Therefore, as $(U_1\cup \{w\}, \{y_1\}, V_1\cup (L-\{y_1\}))$ and $(U_2\cup \{w\}, \{y_2\}, V_2\cup (L-\{y_2\}))$ are both vertical $3$-separations, \blue{and}
$$(V_1\cap V_2)\cup (L-\{y_i\})\cup \{w_i\} \blue{= V_i\cup (L-\{y_i\})},$$
\blue{it follows that $(V_1\cap V_2)\cup (L-\{y_i\})\cup \{w_i\}$} is a cocircuit for each $i\in\{1, 2\}$. Since $y_1\in \cl((V_1\cap V_2)\cup (L-\{y_1\}))$, we have $(V_1\cap V_2)\cup L=\cl(L)$, thereby completing the proof of the lemma.}
\end{proof}


\section{\blue{Theta Separators}}
\label{thetasep}
\blue{
We begin this section by formally defining, for all $n\ge 2$, the matroid $\Theta_n$. Let $n\ge 2$, and let $M$ be the matroid whose ground set is the disjoint union of $W=\{w_1, w_2, \ldots, w_n\}$ and $Z=\{z_1, z_2, \ldots, z_n\}$, and whose circuits are as follows:
\begin{enumerate}[{\rm (i)}]
\item all $3$-element subsets of $W$;
\item all sets of the form $(Z-\{z_i\})\cup \{w_i\}$, where $i\in \{1, 2, \ldots, n\}$; and
\item all sets of the form $(Z-\{z_i\})\cup\{w_j, w_k\}$, where $i$, $j$, and $k$ are distinct elements of $\{1, 2, \ldots, n\}$.
\end{enumerate}
It is shown in \cite[Lemma~2.2]{oxl00} that $M$ is indeed a matroid, and we denote this matroid by $\Theta_n$. If $n=2$, then $\Theta_2$ is isomorphic to the direct sum of $U_{1, 2}$ and $U_{1, 2}$, while if $n=3$, then $\Theta_3$ is isomorphic to $M(K_4)$. Also, for all $n$, the matroid $\Theta_n$ is self-dual under the map that interchanges $w_i$ and $z_i$ for all $i$~\cite[Lemma~2.1]{oxl00}, and the rank of $\Theta_n$ is $n$. For all $i$, we say $w_i$ and $z_i$ are {\em partners}. Furthermore, it is easily checked that, for all $i, j\in \{1, 2, \ldots, n\}$, we have $\Theta_n\delete w_i\cong \Theta_n\delete w_j$. Up to isomorphism, we denote the matroid $\Theta_n\delete w_i$ by $\Theta^-_n$. Observe that if $n=3$, then $\Theta^-_3$ is a $5$-element fan. We refer to the elements in $W$ and $Z$ as the {\em segment elements} and {\em cosegment elements}, respectively, of $\Theta_n$ and $\Theta^-_n$.}

\blue{Recalling the definition of a $\Theta$-separator, the next lemma considers the elasticity of elements in a $\Theta$-separator \blue{when $n\geq 4$}. The analogous lemma for when $n=3$ is covered by Lemma~\ref{elastic fans}. Observe that, if $M$ is $3$-connected and $S$ is a $\Theta$-separator of $M$ such that $M|S\cong \Theta_n$ for some $n\ge 3$, then
$$r(M)=r(M\delete S)+n-2.$$}

\blue{
\begin{lemma}
\label{theta}
Let $M$ be a $3$-connected matroid, and let \blue{$n\ge 4$. Suppose that} $S$ is a $\Theta$-separator of $M$. If $M|S\cong \Theta_n$, then $S$ contains no elastic elements of $M$. Furthermore, if $M|S\cong \Theta^-_n$, then $S$ contains exactly one elastic element, namely the unique cosegment element of $M|S$ with no partner, unless there is an element $w$ of $\cl(S)-S$ such that $M|(S\cup \{w\})\cong \Theta_n$.
\end{lemma}}

\begin{proof}
\blue{Suppose that $M|S\cong \Theta_n$, where $n\ge 4$. Without loss of generality, we may assume that $S$ is the disjoint union of $W=\{w_1, w_2, \ldots, w_n\}$ and $Z=\{z_1, z_2, \ldots, z_n\}$, where $W$ and $Z$ are as defined in the definition of $\Theta_n$. Let $i\in \{1, 2, \ldots, n\}$. As $M|S\cong \Theta_n$, the set $C_i=(Z-\{z_i\})\cup \{w_i\}$ is a circuit of $M$. Now, \blue{as $Z$ has corank~$2$, the circuit $C_i$ has corank~$3$, and so}
$$\blue{\lambda(C_i)=r(C_i)+r^*(C_i)-|C_i|=(|C_i|-1)+3-|C_i|=2.}$$
So $C_i$ is $3$-separating. \blue{Furthermore, $z_i\in \cl^*(C_i)$ and, by Lemma~\ref{orthogonality}, $z_i\not\in \cl(E(M)-(C_i\cup \{z_i\})$. Thus, by Lemma~\ref{3sep2}, $z_i\in \cl^*(E(M)-(C_i\cup \{z_i\})$ and so, as} $E(M)-(C_i\cup \{z_i\})$ contains a triangle in $W-\{w_i\}$,
$$(C_i, \{z_i\}, E(M)-(C_i\cup \{z_i\}))$$
is a cyclic $3$-separation of $M$. \blue{Therefore}, by Lemma~\ref{vertical1}, $z_i$ is not deletable. Moreover, as
$$(Z-\{z_i\}, \{w_i\}, E(M)-((Z-\{z_i\})\cup\{w_i\}))$$
is a vertical $3$-separation of $M$, it follows by Lemma~\ref{vertical1} that $w_i$ is not contractible. Thus $S$ contains no elastic elements of $M$.}

\blue{Now suppose that $M|S\cong \Theta^-_n$, where $n\ge 4$. Without loss of generality, let $S$ be the disjoint union of $W-\{w_j\}$ and $Z$, where $W=\{w_1, w_2, \ldots, w_n\}$ and $Z=\{z_1, z_2, \ldots, z_n\}$ are as defined in the definition of $\Theta_n$. Let $z_i\in Z-\{z_j\}$. Then the argument in the last paragraph shows that
$$((Z-\{z_i\})\cup \{w_i\}, \{z_i\}, E(M)-(Z\cup \{w_i\})$$
is a cyclic $3$-separation of $M$ provided $E(M)-(Z\cup \{w_i\})$ contains a circuit. If $n\ge 5$, then $|W|\ge 4$, and so $E(M)-(Z\cup \{w_i\})$ contains a circuit. Assume that $n=4$. \blue{Then, as} $r^*(M)\ge 4$, we have $|E(M)-(Z\cup \{w_i\})|\ge 3$. Therefore, as $w_k\in \cl(Z\cup \{w_i\})$, where $w_k\in W-\{w_i, w_j\}$, and $Z\cup \{w_i\}$ is exactly $3$-separating, it follows by Lemma~\ref{3sep2} that $w_k\in \cl(E(M)-(Z\cup \{w_i, w_k\})$. In particular, $E(M)-(Z\cup \{w_i\})$ contains a circuit. Hence $z_i$ is not deletable. Furthermore, the argument in the previous paragraph shows that if $w_i\in W-\{w_j\}$, then $w_i$ is not contractible.}

\blue{We complete the proof of the lemma by considering the elasticity of $z_j$. Since $|Z|\geq 4$, it follows by Lemma~\ref{delete1} that $z_j$ is contractible. Assume that $z_j$ is not deletable. Let $i\in \{1, 2, \ldots, n\}$ such that $i\neq j$. Then $C_i=(Z-\{z_i\})\cup \{w_i\}$ is a circuit of $M$. Furthermore,
\begin{align*}
r^*((Z-\{z_i\})\cup \{w_i\}) & = (r(M)-(|C_i|-3))+|C_i|-r(M) \\
& = 3.
\end{align*}
Therefore, as $z_j\in Z-\{z_i\}$ and all elements of $Z-\{z_i\}$ are not deletable, the dual of Lemma~\ref{contract+} implies that there is an element $w$ such that $(Z-\{z_j\})\cup \{w\}$ is a circuit. But then, as $w\in \cl(Z)-Z$, it follows that $w\in \cl(W-\{w_j\})$, and it is easily checked that $M|(S\cup \{w\})\cong \Theta_n$, thereby completing the proof of the lemma.}
\end{proof}

\section{Proofs of Theorem~\ref{main1} and Corollary~\ref{main2}}
\label{proof}

In this section, we prove Theorem~\ref{main1} and Corollary~\ref{main2}. However, almost all of the section consists of the proof of Theorem~\ref{main1}. The proof of this theorem is essentially partitioned into two lemmas, Lemmas~\ref{key2} and~\ref{key3}. Let $M$ be a $3$-connected matroid with a vertical $3$-separation $(X, \{e\}, Y)$ such that $Y\cup \{e\}$ is maximal. Lemma~\ref{key2} establishes Theorem~\ref{main1} for when $X$ contains at least one non-contractible element, while Lemma~\ref{key3} establishes the theorem for when every element in $X$ is contractible.

To prove Lemma~\ref{key2}, we will make use of the following technical result which is extracted from the proof of Lemma 3.2 in~\cite{Oxley et al}.

\begin{lemma}
\label{key1}
Let $M$ be a $3$-connected matroid with a vertical $3$-separation $(X_1, \{e_1\}, Y_1)$ such that $Y_1\cup \{e_1\}$ is maximal. Suppose that $(X_2, \{e_2\}, Y_2)$ is a vertical $3$-separation of $M$ such that $e_2\in X_1$, $e_1\in Y_2$, and $Y_2\cup \{e_2\}$ is closed. Then each of the following holds:
\begin{enumerate}[{\rm (i)}]
\item None of $X_1\cap X_2$, $X_1\cap Y_2$, $Y_1\cap X_2$, and $Y_1\cap Y_2$ are empty.

\item $r((X_1\cap X_2)\cup \{e_2\})=2$.

\item If $|Y_1\cap X_2|=1$, then $X_2$ is a rank-$3$ cocircuit.

\item If $|Y_1\cap X_2|\geq 2$, then $r((X_1\cap Y_2)\cup \{e_1, e_2\})=2$.
\end{enumerate}
\end{lemma}

\begin{lemma}
\label{key2}
Let $M$ be a $3$-connected matroid with a vertical $3$-separation $(X_1, \{e_1\}, Y_1)$ such that $Y_1\cup \{e_1\}$ is maximal. \blue{Suppose that} at least one element of $X_1$ is not contractible. \blue{Then at least one of the following holds:}
\begin{enumerate}[{\rm (i)}]
\item $X_1$ has at least two elastic elements;
\item $X_1\cup \{e_1\}$ is a $4$-element fan; \blue{or}
\item \blue{$X_1$ is contained in a $\Theta$-separator $S$}.
\end{enumerate}
\blue{Moreover, if (iii) holds, then $X_1$ is a rank-$3$ cocircuit, $M^*|S$ is isomorphic to either $\Theta_n$ or $\Theta^-_n$, where $n=|X_1\cup \{e_1\}|-1$, and there is a unique element $x\in X_1$ such that $x$ is a segment element of $M^*|S$ and $(X_1-\{x\})\cup \{e_1\}$ is the set of cosegment elements of $M^*|S$.}
\end{lemma}

\begin{proof}
Let $e_2$ be an element of $X_1$ that is not contractible. Then, by Lemma~\ref{vertical1}, there exists a vertical $3$-separation $(X_2, \{e_2\}, Y_2)$ of $M$. Without loss of generality, we may assume $e_1\in Y_2$. Furthermore, by Lemma~\ref{vertical2}, we may also assume that $Y_2\cup \{e_2\}$ is closed. By Lemma~\ref{key1}, each of $X_1\cap X_2$, $X_1\cap Y_2$, $Y_1\cap X_2$, and $Y_1\cap Y_2$ is non-empty. The proof is partitioned into two cases depending on the size of $Y_1\cap X_2$. Both cases use the following:

\begin{sublemma}
\label{triangle exists}
If $X_1\cap X_2$ contains two contractible elements, then either $X_1$ has at least two elastic elements, or $|X_1\cap X_2|=2$ and there exists a triangle $\{x, y_1, y_2\}$, where $x\in X_1\cap X_2$, $y_1\in Y_1\cap X_2$, and $y_2\in X_1\cap Y_2$.
\end{sublemma}

By Lemma~\ref{key1}(ii), $r((X_1\cap X_2)\cup \{e_2\})=2$. Let $x_1$ and $x_2$ be \blue{distinct} contractible elements of $X_1\cap X_2$. If $|X_1\cap X_2|\ge 3$, then, by Lemma~\ref{delete1} each of $x_1$ and $x_2$ is elastic. Thus we may assume that $|X_1\cap X_2|=2$ and that either $x_1$ or $x_2$, say $x_1$, is not deletable. Let $(U, V)$ be a $2$-separation of $M\delete x_1$ such that neither $r^*(U)=1$ nor $r^*(V)=1$. Since $x_1$ is not deletable, such a separation exists. \blue{Furthermore,} \blue{$|U|, |V|\ge 3$} \blue{as $U$ and $V$ each contain a cycle.} If $x_1\in \cl(U)$ or $x_1\in \cl(V)$, then either $(U\cup \{x_1\}, V)$ or $(U, V\cup \{x_1\})$, respectively, is a $2$-separation of $M$, a contradiction. So $\{x_2, e_2\}\not\subseteq U$ and $\{x_2, e_2\}\not\subseteq V$. Therefore, without loss of generality, we may assume $x_2\in U-\cl(V)$ and $e_2\in V-\cl(U)$. Since $(U, V)$ is a $2$-separation of $M\delete x_1$ and $x_2\not\in \cl(V)$, we deduce that $(U-\{x_2\}, V\cup \{x_1\})$ is a $2$-separation of $M/x_2$. Thus, as $x_2$ is contractible, $\si(M/x_2)$ is $3$-connected, and so $r(U)=2$. In turn, as $Y_1\cup \{e_1\}$ and $Y_2\cup \{e_2\}$ are both closed, this implies that $|U\cap (Y_1\cup \{e_1\})|\le 1$ and $|U\cap (Y_2\cup \{e_2\})|\le 1$; otherwise, $U\subseteq Y_1\cup \{e_1\}$ or $U\subseteq Y_2\cup \{e_2\}$. Thus $|U|=3$ and, in particular, $U$ is the desired triangle. Hence~\blue{(}\ref{triangle exists}\blue{)} holds.

We now distinguish two cases depending on the size of $Y_1\cap X_2$:
\begin{enumerate}[(I)]
\item $|Y_1\cap X_2|=1$; and
		
\item $|Y_1\cap X_2|\geq 2$.
\end{enumerate}

Consider (I). Let $w$ be the unique element in $Y_1\cap X_2$. By Lemma~\ref{key1}, $(X_1\cap X_2)\cup \{e_2\}$ is a segment of at least three elements and $(X_1\cap X_2)\cup \{w\}$ is a rank-$3$ cocircuit. Let $L_1=(X_1\cap X_2)\cup \{e_2\}$. As $|Y_1\cap X_2|=1$, we may assume that $L_1$ is closed.

\begin{sublemma}
\blue{At most one element of $X_1\cap X_2$ is not contractible.}
\label{mostone}
\end{sublemma}

\blue{Suppose that at least two elements in $X_1\cap X_2$ are not contractible, and let $x$ be such an element. Then, by Lemma~\ref{contract+}, there is an element $w'$ distinct from $w$ such that $(L_1-\{x\})\cup \{w'\}$ is a rank-$3$ cocircuit. If $w'\in Y_1$, then $\{w, w'\}\subseteq \cl^*(X_1)$ and $e_1\in \cl(X_1)$, contradicting Lemma~\ref{notmany}. Thus $w'\in X_1$. Since $w'\in \cl^*(L_1-\{x\})$, it follows by Lemma~\ref{3sep1} that each of $(L_1-\{x\})\cup \{w'\}$ and $L_1\cup \{w'\}$ are exactly $3$-separating. Furthermore, as $x\in \cl((L_1-\{x\})\cup \{w'\})$, it follows by Lemma~\ref{3sep2} that $x\not\in \cl^*((L_1-\{x\})\cup \{w'\})$. Therefore
$$((L_1-\{x\})\cup \{w'\}, \{x\}, E(M)-(L_1\cup \{w'\}))$$
is a vertical $3$-separation of $M$. But then, as $L_1\cup \{w'\}\subseteq X_1$, we contradict the maximality of $Y_1\cup \{e_1\}$. Hence~(\ref{mostone}) holds.}

If $|L_1|\ge 4$, then, \blue{by Lemma~\ref{delete1} and~(\ref{mostone}), $L_1-\{e_2\}$, and more particularly $X_1$,} contains at least two elastic elements. Thus, as $|Y_1\cap X_2|=1$, we may assume $|L_1|=3$, and so $(L_1-\{e_2\})\cup \{w\}$ is a triad. Let $L_1=\{x_1, x_2, e_2\}$ and let $\{i, j\}=\{1,2\}$.

\begin{sublemma}
For each $i\in \{1, 2\}$, the element $x_i$ is contractible.
\label{contractx}
\end{sublemma}

If $x_i$ is not contractible, then, by Lemma~\ref{vertical1}, $M$ has a vertical $3$-separation $(U_i, \{x_i\}, V_i)$, where $e_1\in V_i$. By Lemma~\ref{vertical2}, we may assume that $V_i\cup x_i$ is closed. By Lemma~\ref{key1}, $Y_1\cap U_i$ is non-empty and $r((X_1\cap U_i)\cup \{x_i\})=2$. First assume that $|Y_1\cap U_i|=1$. Then $|(X_1\cap U_i)\cup \{x_i\}|\ge 3$, and so $x_i$ is contained in a triangle $T\subseteq (X_1\cap U_i)\cup \{x_i\}$. If $x_j\in V_i$, then, as $V_i\cup \{x_i\}$ is closed, $e_2\in V_i$. Thus $x_j, e_2\not\in T$ and so, by orthogonality, as $\{x_i, x_j, w\}$ is a triad, $w\in T$. This contradicts $w\in Y_1$. It now follows that $x_j\in X_1\cap U_i$ and so $e_2\in X_1\cap U_i$. Thus, as $L_1$ is closed and $L_1\subseteq (X_1\cap U_i)\cup \{x_i\}$,
we have $|(X_1\cap U_i)\cup \{x_i\}|=3$, and therefore $T=\{x_1, x_2, e_2\}$. Let $z$ be the unique element in $Y_1\cap U_i$. Then, by Lemma~\ref{key1} again, $\{\blue{x_j}, e_2, z\}$ is a triad, and so $z\in \cl^*(X_1)$. Furthermore, $w\in \cl^*(X_1)$ and $e_1\in \cl(X_1)$, and so, by Lemma~\ref{notmany}, we deduce that $z=w$. This implies that $Y_2=V_i$. But then $\cl(Y_2\cup \{e_2\})$ contains $x_i$, contradicting that $Y_2\cup \{e_2\}$ is closed. Now assume that $|Y_1\cap U_i|\ge 2$. By Lemma~\ref{key1}, $r((X_1\cap V_i)\cup \{x_i, e_1\})=2$. If $x_j\in V_i$, then, as $V_i\cup \{x_i\}$ is closed, $e_2\in X_1\cap V_i$, and so $\{x_j, e_1, e_2\}$ is a triangle. Since $\{x_1, x_2, w\}$ is a triad, this contradicts orthogonality. Thus $x_j\in \blue{U_i}$. Also, $e_2\in U_i$; otherwise, as $V_i\cup \{x_i\}$ is closed, $x_j\in V_i$, a contradiction. By Lemma~\ref{key1}, $X_1\cap V_i$ is non-empty, and so $M$ has a triangle $T'=\{x_i, e_1, y\}$, where $y\in X_1\cap V_i$. As $\{x_i, x_j, w\}$ is a triad, $T'$ contradicts orthogonality unless $y=w$. But $w\in Y_1$ and therefore cannot be in $X_1\cap V_i$. Hence $x_i$ is contractible, and so~(\ref{contractx}) holds.

Since $x_1$ and $x_2$ are both contractible, it follows by~(\ref{triangle exists}) that either $X_1$ contains two elastic elements or $w$ is in a triangle with two elements of $X_1$. If the latter holds, then $w\in \cl(X_1)$. As $\{x_1, x_2, w\}$ is a triad and $(Y_1\cup \{e_1\})-\{w\}$ is contained in $Y_2\cup e_2$, it follows that $w\not\in \cl((Y_1\cup \{e_1\})-\{w\})$. Therefore
$$(X_1\cup \{w\}, (Y_1\cup \{e_1\})-\{w\})$$
is a $2$-separation of $M$, a contradiction. Thus $X_1$ contains two elastic elements. This concludes (I).

Now consider (II). Let $L_1=(X_1\cap X_2)\cup \{e_2\}$ and $L_2=(X_1\cap Y_2)\cup \{e_1, e_2\}$. By parts (ii) and (iv) of Lemma~\ref{key1}, $L_1$ and $L_2$ are both segments. Since $M$ is $3$-connected, $X_1$ is $3$-separating, and $Y_1\cup \{e_1\}$ is closed, it follows that $X_1$ is a rank-$3$ cocircuit of $M$ \blue{and $L_2$ is closed.}

\blue{First assume that} $|L_2|\ge 4$. \blue{Since $X_1$ is a rank-$3$ cocircuit of $M$, we have $r(Y_1)+1=r(M)$. Therefore, as $|L_2|\ge 4$ and $|X_1\cap X_2|\ge 1$, it follows that $r^*(M)\ge 4$. Now,} \blue{Lemma~\ref{delete1} implies that each element of $L_2$ is deletable.} If $|L_1|\ge 3$, then, by Lemma~\ref{triangle}, each element of $L_2-\{e_1, e_2\}$ is contractible, and so each element of $L_2-\{e_1, e_2\}$ is elastic. Since $|L_2|\ge 4$, it follows that $X_1$ has at least two elastic elements. Thus we may assume that $|L_1|=2$, that is $|X_1\cap X_2|=1$. \blue{We may also assume that $X_1\cap Y_2$ contains at most one contractible element; otherwise, $X_1$ contains at least two elastic elements. Let $e_3, e_4, \ldots, e_n$ denote the elements in $L_1-\{e_1, e_2\}$. Without loss of generality, we may assume that if $X_1\cap Y_2$ contains a contractible element, then it is $e_n$. Let $m=n-1$ if $e_n$ is contractible; otherwise, let $m=n$. Furthermore, let $w_1$ denote the unique element in $X_1\cap X_2$. Since $(L_2-\{e_1\})\cup \{w_1\}$ is a rank-$3$ cocircuit, and at most one element of $L_2-\{e_1\}$ is contractible, it follows by Lemma~\ref{contract+} that, for all $i\in \{2, 3, \ldots, m\}$, there are distinct elements $w_2, w_3, \ldots, w_m$ of $Y_1$ such that $(L_2-\{e_i\})\cup \{w_i\}$ is a cocircuit. Let $W=\{w_1, w_2, \ldots, w_m\}$. As $W$ is in the coclosure of the $3$-separating set $L_2$, we have $r^*(W)=2$. It follows that $(L_2-\{e_i\})\cup \{w_j, w_k\}$ is a cocircuit of $M$ for all distinct elements $i, j, k\in \{1, 2, \ldots, m\}$. By a comparison of the circuits of $\Theta_n$, it is straightforward to deduce that $M^*|(W\cup L_2)$ is isomorphic to either $\Theta_n$ if no element of $X_1\cap Y_2$ is contractible, or $\Theta^-_n$ if $e_n$ is contractible. Hence $X_1$ is contained in a $\Theta$-separator of $M$ as described in the statement of the lemma.}

We may now assume that $|L_2|=3$.
Let $L_2=\{e_2, a, e_1\}$. If $|X_1\cap X_2|=1$, then $|X_1|=3$, and so $X_1$ is a triad. In turn, this implies that $X_1\cup \{e_1\}$ is a $4$-element fan. Thus $|X_1\cap X_2|\ge 2$.
Let $x_1$ and $x_2$ be distinct elements in $X_1\cap X_2$. Since $\{e_1, a, e_2\}$ is a triangle in $M/x_i$ for each $i\in\{1,2\}$, it follows by Lemma~\ref{triangle} that $x_i$ is contractible for each $i\in\{1,2\}$. Thus, by~(\ref{triangle exists}), either $X_1$ contains two elastic elements, or $X_1\cap X_2=\{x_1, x_2\}$ and $a$ is in a triangle with two elements of $X_2$. The latter implies that $a\in \cl(X_2\cup \{e_2\})$. As $a\not\in \cl(Y_1\cup \{e_1\})$ and $Y_2-\{a\}$ is contained in $Y_1\cup \{e_1\}$, it follows that $a\not\in \cl(Y_2-\{a\})$. Hence, \blue{as
$$r(X_2\cup \{e_2\})+r(Y_2)-r(M)=2,$$
we have $r(X_2\cup \{e_2, a\})+r(Y_2-\{a\})+1-r(M)=2$, and so}
$$(X_2\cup \{a, e_2\},Y_2-\{a\})$$
is a $2$-separation of $M$, a contradiction. Thus $X_1$ contains two elastic elements. This concludes (II) and the proof of the lemma.
\end{proof}

\begin{lemma}
\label{key3}
Let $M$ be a $3$-connected matroid with a vertical $3$-separation $(X_1, \{e_1\}, Y_1)$ such that $Y_1\cup \{e_1\}$ is maximal. \blue{Suppose that} every element of $X_1$ is contractible. \blue{Then at least one of the following holds:}
\begin{enumerate}[{\rm (i)}]
\item $X_1$ has at least two elastic elements;

\item $X_1\cup \{e_1\}$ is a $4$-element fan; \blue{or}

\item \blue{$X_1$ is contained in a $\Theta$-separator $S$}.
\end{enumerate}
\blue{Moreover, if (iii) holds, then $X_1\cup \{e_1\}$ is a circuit, $M|S$ is isomorphic to either $\Theta_n$ or $\Theta^-_n$ for some $n\in \{|X_1|, |X_1|+1\}$, and $X_1$ is a subset of the cosegment elements of $M|S$.}
\end{lemma}

\begin{proof}
First suppose that $X_1$ is independent. Then, as \blue{$r(X_1)=|X_1|$ and $\lambda(X_1)=r(X_1)+r^*(X_1)-|X_1|$, we have} $r^*(X_1)=2$. That is, $X_1$ is a segment in $M^*$. \blue{As $r^*(X_1)=2$, it follows that either $(X_1-\{x\})\cup \{e_1\}$ is a circuit for some $x\in X_1$, or $X_1\cup \{e_1\}$ is a circuit. If $(X_1-\{x\})\cup \{e_1\}$ is a circuit, then either $X_1\cup \{e_1\}$ is a $4$-element fan, or it is easily checked that $(X_1-\{x\}, \{e_1\}, Y_1\cup \{x\})$ is a vertical $3$-separation, contradicting the maximality of $Y_1\cup \{e_1\}$. Thus we may assume that $X_1\cup \{e_1\}$ is a circuit of $M$. Now, if two elements of $X_1$ are deletable, then $X_1$ contains at least two elastic elements, so we may assume that at most one element of $X_1$ is deletable. Assume first that $X_1$ is coclosed, and let $X_1=\{z_1, z_2, \ldots, z_n\}$. Without loss of generality, we may assume that if $X_1$ contains a deletable element, then it is $z_n$. Let $m=n-1$ if $z_n$ is deletable; otherwise, let $m=n$. Since $X_1\cup \{e_1\}$ has corank~$3$ and $X_1$ is coclosed, it follows by the dual of Lemma~\ref{contract+} that, for all $i\in \{1, 2, \ldots, m\}$, there are distinct elements $w_1, w_2, \ldots, w_m$ such that $(X_1-\{z_i\})\cup \{w_i\}$ is a circuit. Let $W=\{w_1, w_2, \ldots, w_m\}$. Since $X_1$ is $3$-separating and $W\subseteq \cl(X_1)$, it follows that $r(W)=2$. As every $3$-element subset of $X_1$ is a cocircuit, it follows by orthogonality that $(X_1-\{z_i\})\cup \{w_j, w_k\}$ is a circuit for all distinct $i, j, k\in \{1, 2, \ldots, m\}$. By a comparison with the circuits of $\Theta_n$, it is easily checked that $M|(W\cup X_1)$ is isomorphic to $\Theta_n$ if $m=n$, and $M|(W\cup X_1)$ is isomorphic to $\Theta^-_n$ if $m=n-1$, and so $X_1$ is contained in a $\Theta$-separator of $M$ as described in the statement of the lemma. Now assume that $X_1$ is not coclosed. Then, as $X_1\cup \{e_1\}$ is a corank-$3$ circuit, $|\cl^*(X_1)-X_1|=1$. Let $\{z_1\}=\cl^*(X_1)-X_1$, and denote the elements of $X_1$ as $z_2, z_3, \ldots, z_n$. Applying the previous argument to $X_1\cup \{z_1\}$ and recalling that $X_1\cup \{e_1\}$ is a circuit, we deduce that $X_1$ is again contained in a $\Theta$-separator of $M$ as described in the statement of the lemma.}

Now suppose that $X_1$ is dependent, and let $C$ be a circuit in $X_1$. As $M$ is $3$-connected, $|C|\ge 3$. If every element in $C$ is deletable, then $X_1$ contains at least two elastic elements. Thus we may assume that there is an element, \blue{say $g$}, in $C$ that is not deletable. By Lemma~\ref{vertical1}, there exists a cyclic $3$-separation $(U, \{g\}, V)$ in $M$, where $e_1\in V$. By Lemma~\ref{vertical2}, we may also assume that $V\cup \{g\}$ is coclosed. Note that, as $(U, \{g\}, V)$ is a cyclic $3$-separation, $r^*(U)\ge 3$, and so $|U|\ge 3$.

We next show that

\begin{sublemma}
$|X_1\cap U|, |X_1\cap V|\ge 2$.
\label{least2}
\end{sublemma}

If either $C-\{g\}\subseteq U$ or $C-\{g\}\subseteq V$, then $g\in \cl(U)$ or $g\in \cl(V)$, respectively, \blue{in which case either $(U\cup \{g\}, V)$ or $(U, V\cup \{g\})$ is a $2$-separation of $M$,} a contradiction. Thus $C\cap (X_1\cap U)$ and $C\cap (X_1\cap V)$ are both non-empty, and so $|X_1\cap U|, |X_1\cap V|\ge 1$. Say $X_1\cap U=\{g'\}$, where $g'\in C$. Since $C$ is a circuit, $g\in \cl_{M/g'}(V)$. Therefore, as $Y_1\cup \{e_1\}$ is closed and so $g'\not\in \cl(Y_1)$, and $(U, V)$ is a $2$-separation of $M\delete g$, we have
\begin{align*}
\lambda_{M/g'}(U\cap Y_1)&=r_{M/g'}(U\cap Y_1)+r_{M/g'}(V\cup \{g\}) -r(M/g') \\
& = r_M(U\cap Y_1)+r_M(V)-(r(M)-1) \\
& = r_M(U\cap Y_1)+r_M(V)-r(M\delete g)+1 \\
& = r_M(U)-1+r_M(V)-r(M\delete g)+1 \\
& = r_M(U)+r_M(V)-r(M\delete g) \\
& =1.
\end{align*}
Thus $(U\cap Y_1, V\cup \{g\})$ is a $2$-separation of $M/g'$. Since every element in $X_1$ is contractible, $g'$ is contractible, and so $r(U)=2$. Since $|U|\ge 3$, it follows that $|U\cap Y_1|\ge 2$, and so $g'\in \cl(Y_1\cup \{e_1\})$, a contradiction as $Y_1\cup \{e_1\}$ is closed. Hence $|X_1\cap U|\ge 2$. An identical argument interchanging the roles of $U$ and $V$ establishes that $|X_1\cap V|\ge 2$, thereby establishing~(\ref{least2}).

\blue{Say $|Y_1\cap U|\ge 2$. It follows by two application of uncrossing that each of $(X_1\cap V)\cup\{g\}$ and $(X_1\cap V)\cup\{g,e_1\}$ is $3$-separating. Since $|X_1\cap V|\ge 2$ and $M$ is $3$-connected, $(X_1\cap V)\cup\{g\}$ and $(X_1\cap V)\cup\{g, e_1\}$ are exactly $3$-separating. Therefore, by Lemma~\ref{3sep1}, $e_1\in \cl((X_1\cap V)\cup\{g\})$ or $e_1\in \cl^*((X_1\cap V)\cup\{g\})$. Since $e_1\in \cl(Y_1)$, it follows by Lemma~\ref{orthogonality} that $e_1\not\in \cl^*((X_1\cap V)\cup \{g\})$. So $e_1\in \cl((X_1\cap V)\cup\{g\})$. Thus, if $r((X_1\cap V)\cup\{g\})\ge 3$, then $((X_1\cap V)\cup\{g\}, \{e_1\}, Y_1\cup U)$ is a vertical $3$-separation, contradicting the maximality of $Y_1\cup \{e_1\}$. Therefore $r((X_1\cap V)\cup\{e_1,g\})=2$. But then $g\in \cl(V\cap X_1)\subseteq \cl(V)$, a contradiction.}
	
\blue{Now assume that $|Y_1\cap U|\le 1$.}  Say $Y_1\cap U$ is empty. Then $U\subseteq X_1$. Let $(U', \{h\}, V')$ be a cyclic $3$-separation of $M$ such that $V\cup \{g\}\subseteq V'\cup \{h\}$ with the property that there is no other cyclic $3$-separation $(U'', \{h'\}, V'')$ in which $V'\cup \{h\}$ is a proper subset of $V''\cup \{h'\}$. Observe that such a cyclic $3$-separation exists as we can choose $(U, \{g\}, V)$ if necessary. If every element in $U'$ is deletable, then, as $U'\subseteq X_1$ and $|U'|\ge 3$, it follows that $X_1$ has at least two elastic elements. Thus we may assume that there is an element in $U'$ that is not deletable. By the dual of Lemma~\ref{key2}, either $U'$, and thus $X_1$, contains at least two elastic elements or $U'\cup \{h\}$ is a $4$-element fan, \blue{or $U'$ is contained in a $\Theta$-separator}. If \blue{$U'\cup \{h\}$ is a $4$-element fan,} then, by Lemma~\ref{fans & seps},
$$((U'\cup \{h\})-\{f\}, \{f\}, E\blue{(M)}-(U'\cup \{h\}))$$
is a vertical $3$-separation, where $f$ is the spoke-end of the $4$-element fan $U'\cup \{h\}$. But then, as $X_1\cap V$ is non-empty, $Y_1\cup \{e_1\}$ is properly contained in $E\blue{(M)}-(U'\cup \{h\})$, contradicting maximality. \blue{If $U'$ is contained in a $\Theta$-separator, then, by the dual of Lemma~\ref{key2}, $U'$ is a circuit and there is an element $w$ of $U'$ such that $(U'-\{w\})\cup\{h\}$ is a cosegment. But then
$$((U'\cup \{h\})-\{w\}, \{w\}, E(M)-(U'\cup \{h\}))$$
is a vertical $3$-separation of $M$, contradicting the maximality of $Y_1\cup \{e_1\}$ as $Y_1\cup \{e_1\}$ is properly contained in $E(M)-(U'\cup \{h\})$.} Hence we may assume that $|Y_1\cap U|=1$.

Let $Y_1\cap U=\{y\}.$ Since $|Y_1\cap U|=1$, we have $|Y_1\cap V|\ge 2$ and so, by two applications of uncrossing, $X_1\cap U$ and $(X_1\cap U)\cup \{g\}$ are both $3$-separating. \blue{Since $M$ is $3$-connected and $|X_1\cap U|\ge 2$, these sets are exactly $3$-separating.} If $y\not\in \cl(X_1\cap U)$, then, by Lemma~\ref{orthogonality}, $y\in \cl^*(V\cup \{g\})$. But then $V\cup \{g\}$ is not coclosed, a contradiction. Thus $y\in \cl(X_1\cap U)$, and so $y\in \cl((X_1\cap U)\cup \{g\})$. \blue{Now} $y\not\in \cl^*(V\cup \{g\})$, and so $y\not\in \cl^*(V)$. \blue{Hence as $(X_1\cap U)\cup \{g\}$ and, therefore, the complement $V\cup \{y\}$ is $3$-separating,} Lemma~\ref{3sep1} implies that $y\in \cl(V)$. Therefore, as $(X_1\cap U)\cup \{g\}$ and $V$ each have rank at least three, it follows that $((X_1\cap U)\cup \{g\}, \{y\}, V)$ is a vertical $3$-separation of $M$. \blue{Note that $r(V)\ge 3$; otherwise, $(X_1\cap V)\subseteq \cl(\{y, e_1\})$, in which case, $Y_1\cup \{e_1\}$ is not closed.} But $(X_1\cap U)\cup \{g\}$ is a proper subset of $X_1$, a contradiction to the maximality of $Y_1\cup \{e_1\}$. This last contradiction completes the proof of the lemma.
\end{proof}

We now combine Lemmas~\ref{key2} and~\ref{key3} to prove Theorem~\ref{main1}.

\begin{proof}[Proof of Theorem~\ref{main1}]
Let $(X, \{e\}, Y)$ be a vertical $3$-separation of $M$, where $Y\cup \{e\}$ is maximal, and suppose that $X\cup \{e\}$ is not a $4$-element fan \blue{and $X$ is not contained in a $\Theta$-separator}. If at least one element in $X$ is not contractible, then, by Lemma~\ref{key2}, $X$ contains at least two elastic elements. On the other hand if every element in $X$ is contractible, then by Lemma~\ref{key3}, $X$ again contains at least two elastic elements. \blue{This completes the proof of the theorem.}
\end{proof}

We end the paper by establishing Corollary~\ref{main2}.
\begin{proof}[Proof of Corollary~\ref{main2}]
Let $M$ be a $3$-connected matroid. If every element of $M$ is elastic, then the corollary holds. Therefore suppose that $M$ has at least one non-elastic element, $e$ say. Up to duality, we may assume that $\si(M/e)$ is not $3$-connected. Then, by Lemma~\ref{vertical1}, $M$ has a vertical $3$-separation $(X, \{e\}, Y)$. As $r(X), r(Y)\ge 3$, this implies that $|E(M)|\ge 7$, and so we deduce that every element in a \blue{$3$-connected} matroid with at most six elements is elastic. Now, \blue{suppose that $M$ has no $4$-element fans and no $\Theta$-separators, and} let $(X', \{e'\}, Y')$ be a vertical $3$-separation such that $Y'\cup \{e'\}$ is maximal and contains $Y\cup \{e\}$. Then it follows by Theorem~\ref{main1} that $X'$, and hence $X$, contains at least two elastic elements. Interchanging the roles of $X$ and $Y$, \blue{an identical argument} gives us that $Y$ also contains at least two elastic elements. Thus, $M$ contains at least four elastic elements.
\end{proof}

\section*{Acknowledgments} The authors thank the referee for their comments. The fourth author was supported by the New Zealand Marsden Fund.

\end{document}